\documentclass[11pt,leqno]{amsart}
\usepackage{verbatim,amssymb,amsfonts,latexsym,amsmath,amsthm,bbm,nicefrac,xspace,enumerate,tikz,marvosym,mathtools}
\usetikzlibrary{decorations.pathmorphing,shapes}

\newtheorem{theorem}{Theorem}[section]
\newtheorem*{theorem*}{Theorem}
\newtheorem*{corollary*}{Corollary}
\newtheorem*{maintheorem}{Main Theorem}

\newtheorem{question}[theorem]{Question}

\newtheorem*{claim}{Claim}
\theoremstyle{definition}

\theoremstyle{remark}

\newcommand{\N}{\mathbb{N}}

\renewcommand{\P}{\mathcal{P}}

\newcommand{\explicitSet}[1]{\left\lbrace #1 \right\rbrace}
\newcommand{\brackets}[1]{\left\langle #1 \right\rangle}
\newcommand{\set}[2]{\explicitSet{#1 \colon #2}}
\newcommand{\seq}[2]{\brackets{#1 \colon #2}}
\newcommand{\<}{\langle}
\renewcommand{\>}{\rangle}
\renewcommand{\a}{\alpha}
\renewcommand{\b}{\beta}

\newcommand{\dlt}{\delta}

\newcommand{\z}{\zeta}
\renewcommand{\k}{\kappa}

\newcommand{\w}{\omega}
\newcommand{\0}{\emptyset}

\newcommand{\sub}{\subseteq}
\newcommand{\rest}{\!\restriction\!}
\newcommand{\cat}{\!\,^{\frown}}

\newcommand{\cf}{\mathrm{cf}}
\newcommand{\card}[1]{\left\lvert #1 \right\rvert}

\newcommand{\PP}{\mathbb{P}}
\newcommand{\forces}{\Vdash}

\newcommand{\pwmf}{\nicefrac{\mathcal{P}(\w)}{\mathrm{Fin}}}

\newcommand{\continuum}{\mathfrak{c}}

\newcommand{\dom}{\mathfrak d}

\newcommand{\ch}{\ensuremath{\mathsf{CH}}\xspace}
\newcommand{\gch}{\ensuremath{\mathsf{GCH}}\xspace}
\newcommand{\zfc}{\ensuremath{\mathsf{ZFC}}\xspace}

\newcommand{\oca}{\ensuremath{\mathsf{OCA}}\xspace}

\begin{document}

\title[Nontrivial automorphisms of $\mathcal P(\omega)/\mathrm{Fin}$ in Cohen models]{Nontrivial automorphisms of $\mathcal P(\omega)/\mathrm{Fin}$ \\ in Cohen models}
\author{Will Brian}
\address {
W. R. Brian\\
Department of Mathematics and Statistics\\
University of North Carolina at Charlotte\\
Charlotte, NC 
(USA)}
\email{wbrian.math@gmail.com}
\urladdr{wrbrian.wordpress.com}
\author{Alan Dow}
\address {
A. S. Dow\\
Department of Mathematics and Statistics\\
University of North Carolina at Charlotte\\
Charlotte, NC 
(USA)}
\email{adow@charlotte.edu}
\urladdr{https://webpages.uncc.edu/adow}

\subjclass[2020]{03E35, 08A35, 54D40}
\keywords{automorphisms of $\mathcal P(\omega)/\mathrm{Fin}$, Cohen forcing, Davies trees, elementary submodels, consequences of Jensen's $\square$, weak Freese-Nation property}

\thanks{The first author is supported in part by NSF grant DMS-2154229}

\begin{abstract}
We show that if $\kappa < \aleph_\w$ Cohen reals are added to a model of $\mathsf{CH}$, then there are nontrivial automorphisms of $\mathcal P(\omega)/\mathrm{Fin}$ in the extension. 
Under some further hypotheses on the ground model, namely the existence of long enough sage Davies trees (which follows from $\mathsf{SCH}$ plus $\square_\lambda$ for every $\lambda$ with $\mathrm{cf}(\lambda) = \omega$), we prove the same result for cardinals $\kappa \geq \aleph_\w$ as well. 
This extends a result a Shelah and Stepr\={a}ns, who proved the result for $\kappa = \aleph_2$. 
\end{abstract}

\maketitle

\section{Introduction}

Given a cardinal $\k \geq \aleph_2$, the ``$\k$-Cohen model'' means any model obtained by forcing over a model of \ch with the poset $\mathrm{Fn}(\k,2)$ of finite partial functions $\k \to 2$, the usual poset for adding $\k$ mutually generic Cohen reals. 
In this paper we look at the question of whether the $\k$-Cohen model contains nontrivial automorphisms of the Boolean algebra $\pwmf$. 
We prove that the answer is yes if $\kappa < \aleph_\omega$, and the answer is yes for all $\k$ under some additional ``$L$-like'' hypotheses on the ground model. 
Furthermore, if $\k$ is regular then not only do these models contain nontrivial automorphisms, but they contain $2^\continuum$ of them, the maximum possible number. 

\begin{maintheorem}
Assume \ch, let $\k \geq \aleph_2$, and let $\PP = \mathrm{Fn}(\k,2)$. 
\begin{enumerate} 
\item If $\k < \aleph_\w$, then every automorphism of $\pwmf$ in $V$ extends to $2^\k$ automorphisms of $\pwmf$ in $V^\PP$. 
\item More generally, if $\k$ is a regular cardinal and there is a sage Davies tree of length $\k$, then every automorphism of $\pwmf$ in $V$ extends to $2^\k$ automorphisms of $\pwmf$ in $V^\PP$.
\item If $\k$ is a singular cardinal and there is a sage Davies tree of length $\nu = \k^\w$, then there are nontrivial automorphisms of $\pwmf$ in $V^\PP$. 
\end{enumerate}
\end{maintheorem} 

A Davies tree is a sequence $\seq{M_\a}{\a < \k}$ of elementary submodels of some large fragment $H_\theta$ of the set-theoretic universe such that the $M_\a$ enjoy certain coherence and covering properties. (These sequences are called ``trees'' because they are constructed by cleverly enumerating the leaves of a tree of elementary submodels of $H_\theta$.) These structures provide a unified framework for carrying out a wide variety of constructions in infinite combinatorics. They were introduced by R. O. Davies in \cite{Davies}, and an excellent survey of their many uses can be found in \cite{Soukups}. 
A version of Davies trees using countably closed models, called sage Davies trees, is introduced in \cite{Soukups}. 
Their construction uses the singular cardinals hypothesis $\mathsf{SCH}$ below $\k$ (the length of the sequence), plus $\square_\lambda$ for limit cardinals $\lambda < \k$ of countable cofinality. 
As these hypotheses are vacuously true for $\k < \aleph_\w$, sage Davies trees of length $\k<\aleph_\w$ can be proved to exist from \zfc alone. 
In particular, part $(1)$ of our main theorem is a direct consequence of $(2)$; we have stated it separately simply to emphasize that $(1)$ is a theorem of \zfc, while the more general $(2)$ requires an additional hypothesis when $\k \geq \aleph_\w$. 
As shown in \cite{Soukups}, it is consistent relative to a supercompact cardinal that there are no sage Davies trees of length $\geq \aleph_\w$ (a large cardinal hypothesis is necessary because the failure of either $\mathsf{SCH}$ or of $\square_\lambda$ has large cardinal strength). 
We do not know whether the general conclusion of $(2)$ is a theorem of \zfc.


The $\kappa = \aleph_2$ case of our main theorem was proved by 
Shelah and Stepr\={a}ns in \cite{SS1}. 
(Their theorem states only that the $\aleph_2$-Cohen model contains nontrivial automorphisms, but their argument can be modified to obtain $2^{\aleph_2}$ automorphisms extending any given ground model automorphism.) 
However, the Shelah-Stepr\={a}ns proof does not work for $\k \geq \aleph_3$. 

Roughly, the reason is that $\pwmf$ has a special structure in the $\aleph_2$-Cohen model, a kind of ``near-saturation'' that enables some \ch-like constructions, or modifications of them, to go through in the $\aleph_2$-Cohen model (see \cite{Step} or \cite{DH}). 
(Recall that $\pwmf$ is a countably saturated Boolean algebra, as shown in \cite{Par} and \cite{JO}. 
Under \ch, where $\pwmf$ can be written as an increasing union of countable subalgebras, this enables us to recursively construct many automorphisms of $\pwmf$.) 
The Shelah-Stepr\={a}ns result in \cite{SS1} predates \cite{Step} and \cite{DH}, but in hindsight, their argument utilizes aspects of this special structure in the $\aleph_2$-Cohen model. 
And these nice features of the $\aleph_2$-Cohen model that drive the Shelah-Stepr\={a}ns proof are simply lacking in the $\k$-Cohen model for larger $\k$. 

In a bit more detail, the Shelah-Stepr\={a}ns argument utilizes the fact that if $M \models \zfc$, and $g$ is a Cohen-generic real over $M$, then every real $x \in M[g]$ cuts $\pwmf \cap M$ in an $(\w,\w)$-gap. 
This means that there are countable sets $L,U \sub \pwmf \cap M$ such that $x$ is above every member of $L$ and below every member of $U$, and meets every member of $\pwmf \cap M$ that is not below some $\ell \in L$ or above some $u \in U$. 
This fact is important when trying to build an automorphism of $\pwmf$. To extend a partial automorphism $\Phi$ with domain $D$ to some $x \notin D$, we need to find a $y \in \pwmf$ such that if $x$ cuts $D$ in the gap $(L,U)$, then $y$ cuts $\Phi[D]$ in the gap $(\Phi[L],\Phi[U])$. 
The natural forcing poset to fill an $(\w,\w)$-gap in $\pwmf$ is countable, hence forcing equivalent to the Cohen poset $\mathrm{Fn}(\w,2)$. 
Therefore, if we have access to a real $h$ Cohen-generic over $\Phi$, and $D$, then a generic for the poset that fills $(\Phi[L],\Phi[U])$ can be defined from $h$. 

Now let us consider the $\aleph_2$-Cohen model $V[G]$. 
Because $\mathrm{Fn}(\w_2,2)$ is equivalent to a length-$\w_2$ finite support iteration of $\mathrm{Fn}(\w_1,2)$, we can write $V[G]$ as an increasing union $V[G] = \bigcup_{\a < \w_2}V[G_\a]$, where each $V[G_\a]$ is a model of \ch and contains $\aleph_1$ ``fresh'' Cohen-generic reals. 
To build an automorphism $\Phi$ in $V[G]$, we begin with some automorphism $\Phi_0$ of $\pwmf$ in $V$ and recursively extend it to each $V[G_\a]$. At successor steps, to extend $\Phi$ from $\pwmf \cap V[G_\a]$ to $\pwmf \cap V[G_{\a+1}]$, we do the following: 
\begin{itemize}
\item[$\circ$] Using the fact that $V[G_{\a+1}] \models \ch$, enumerate $\pwmf \cap V[G_{\a+1}]$ in order type $\w_1$, say $\pwmf \cap V[G_{\a+1}] = \{b_\xi :\, \xi < \w_1\}$. 
\item[$\circ$] Stepping through the $b_\xi$ one at a time, at each step we add one or two new reals to the domain of $\Phi$, and close the domain under Boolean operations. So at step $\xi$, the domain of $\Phi$ is countably generated over $\pwmf \cap V[G_\a]$. From this and the comments above, it follows that the new real $b_\xi$ cuts the domain of $\Phi$ in an $(\w,\w)$-gap $(L,U)$.
\item[$\circ$] Because $V[G_{\a+1}]$ contains $\aleph_1$ ``fresh'' Cohen reals over $V[G_\a]$, some of these are still untouched and can be used to fill the $(\w,\w)$-gap $(\Phi[L],\Phi[U])$ and find an image for $b_\xi$. After adding $b_\xi$ to the domain of $\Phi$, it can be added to the image of $\Phi$ in a similar fashion.
\end{itemize}

If $\k \geq \aleph_3$ then this argument breaks down. In this case $V[G]$ is no longer a union of intermediate models $V[G_\a]$ satisfying \ch. 
Yet in the preceding argument, it was essential that when extending $\Phi$ from $V[G_\a]$ to $V[G_{\a+1}]$ that the new reals be enumerated in type $\w_1$. 
Without this, we cannot keep the domain of $\Phi$ countably generated over $\pwmf \cap V[G_\a]$, and there is no reason to think a new real $b_\xi$ should cut the domain of $\Phi$ in an $(\w,\w)$-gap. 

In order to extend this argument to higher Cohen models, we need a sequence $\seq{B_\a}{\a < \k}$ of subalgebras of $\pwmf$ such that 
\vspace{-2mm}
\begin{itemize}
\item[$\circ$] each $B_{\a+1}$ is $\w_1$-generated over $B_\a$ (so the generators can be dealt with in order type $\w_1$), and 
\item[$\circ$] each $B_{\a+1}$ is able to generically fill all the $(\w,\w)$-gaps that arise in the process of extending our automorphism. 
\end{itemize}

The main insight of this paper is that such a sequence of subalgebras is provided by forcing over a sage Davies tree. 
Given a sage Davies tree $\seq{M_\a}{\a < \k}$ in the ground model, we set $B_\a = \pwmf \cap \bigcup_{\xi < \a}M_\xi[G]$. 
It is not too hard to show that each $B_{\a+1}$ is $\w_1$-generated over $B_\a$ and contains $\aleph_1$ ``fresh'' Cohen-generic reals over $B_\a$. 
The difficulty lies in making sure that $M_{\a+1}[G]$ is able to utilize these Cohen reals to fill all the $(\w,\w)$-gaps that arise while trying to extend $\Phi$ from $B_\a$ to $B_{\a+1}$. 
For this to happen, we must ensure these gaps, and their images under $\Phi$, are members of $M_{\a+1}[G]$, since then $M_{\a+1}[G]$ can compute sufficiently generic filters on the required gap-filling posets. 

Section 2 describes what becomes of a sage Davies tree after a $\k$-Cohen forcing. While the resulting sequence of models is no longer a sage Davies tree, it remembers enough of its original structure to provide a useful combinatorial tool in the extension. 
In Section 3 we show how to use this tool to extend the Shelah-Stepr\={a}ns argument from \cite{SS1} and obtain many automorphisms of $\pwmf$. 
Using remnants of sage Davies trees in Cohen models was first explored in \cite{BDS}, though in a rather different way. 
Also, let us point out that Theorem~\ref{thm:firewood}(2) is reminiscent of \cite{FS,FKS}, where it is shown that $\pwmf$ has the weak Freese-Nation property in $\k$-Cohen models under essentially the same hypotheses as in our main theorem.

\section{Sage Davies trees after Cohen forcing}

In what follows, $H_\theta$ denotes the set of all sets hereditarily smaller than some big enough cardinal $\theta$. (Taking $\theta = (2^\k)^+$ suffices in all that follows.) Given two sets $M$ and $N$, we write $M \prec N$ to mean that $(M,\in)$ is an elementary submodel of $(N,\in)$. A set $M$ is \emph{countably closed} if $M^\w \sub M$. If $M$ satisfies (enough of) \zfc, this is equivalent to the property $[M]^\w \sub M$.


Let $\k$ be a cardinal with $\cf(\k) > \w$. A \emph{sage Davies tree for $\k$} over a set $P$ (a parameter) is a sequence $\seq{M_\a}{\a < \k}$ of elementary submodels of $H_\theta$, for some ``big enough'' regular cardinal $\theta$, such that 
\begin{enumerate}
\item Each $M_\a$ is countably closed, and $|M_\a| = \aleph_1$, and $P \in M_\a$. 
\item $[\k]^\w \sub \bigcup_{\a < \k}M_\a$. 
\item For each $\a < \k$, there is a countable collection $\mathcal N_\a$ of countably closed elementary submodels of $H_\theta$, each containing $P$, such that 
$$\textstyle \bigcup_{\xi < \a}M_\xi \,=\, \bigcup \mathcal N_\a.$$
\item For each $\a < \k$, $\seq{M_\xi}{\xi < \a} \in M_\a$. 
\item $\bigcup_{\a < \k}M_\a$ is a countably closed elementary submodel of $H_\theta$. 
\end{enumerate} 
\begin{enumerate}
\item[$(6)$] Furthermore, if $\k$ is a regular cardinal then 
$$\textstyle \set{\dlt \in \k\,}{\dlt = \k \cap \bigcup_{\xi < \dlt}M_\xi}$$ 
is a closed and unbounded subset of $\k$. 
\end{enumerate} 
The definition of sage Davies trees in \cite{Soukups} lists only properties $(1)-(5)$. We have taken the liberty of adding $(6)$ into the definition because it follows easily from the construction of the sage Davies trees in \cite[Section 14]{Soukups}. 
Also, though we have stated the full definition from \cite{Soukups} in $(1)-(5)$, property $(5)$ is not used anywhere in what follows. 

\smallskip

Let $\mathrm{Fn}(A,2)$ denote the usual Cohen forcing consisting of all the finite partial functions $A \to 2$, ordered by extension. 
Let $\mathrm{supp}(q)$ denote the domain of a condition $q \in \mathrm{Fn}(A,2)$, and if $\dot x$ is an $\mathrm{Fn}(A,2)$-name then let $\mathrm{supp}(\dot x) = \bigcup \set{\mathrm{supp}(q)}{q \in \mathrm{Fn}(A,2) \cap \mathrm{TC}(\dot x)}$ (where $\mathrm{TC}(\dot x)$ denotes the transitive closure of $\dot x$).

Given a forcing poset $\PP$, recall that a \emph{nice $\PP$-name} for a subset of $\w$ is a subset $\dot x$ of $\check \w \times \PP$ such that for each $n \in \N$, $\set{p}{(\check n,p) \in \dot x}$ is an antichain in $\PP$. 
For convenience, we will write $n$ for $\check n$, and consider a nice name to be a subset of $\w \times \PP$. 
The evaluation of a nice name $\dot x$ with respect to a filter $G \sub \PP$ is denoted 
$\mathrm{val}_G(\dot x) = \set{n \in \w}{(n,p) \in \dot x \text{ for some } p \in G}$. 

\begin{theorem}\label{thm:firewood}
Let $\k$ be an infinite cardinal with $\k^\w = \k$, 
and suppose $\seq{M_\a}{\a < \k}$ is a sage Davies tree for $\k$. 
If $G$ is $\mathrm{Fn}(\k,2)$-generic over $V$, then the sequence $\seq{M_\a[G]}{\a < \k}$ in $V[G]$ satisfies: 
\begin{enumerate}
\item $\card{\mathrm{Lim}(\k) \cap M_\a \setminus \bigcup_{\xi < \a}M_\xi} = \aleph_1$ for each $\a < \k\vphantom{f^f}$. 
Consequently, there is a function $C: \mathrm{Lim}(\k) \to \P(\w)$ 
such that for all $\a < \k$, $C \in M_\a[G]$ and $C$ maps $\mathrm{Lim}(\k) \cap M_\a$ to a set of $\aleph_1$ reals mutually Cohen generic over $V\big[ G \cap \bigcup_{\xi<\a}M_\xi \big]$.
\item For each $\vphantom{f^{f^{f}}} \a < \k$ and $r \in \pwmf \cap M_\a[G] \setminus \bigcup_{\xi < \a} M_\xi[G]$, there are countable sets $L^r,U^r \sub \pwmf \cap \bigcup_{\xi < \a} M_\xi[G]$ 
such that for every $x \in \pwmf \cap \bigcup_{\xi < \a} M_\xi[G]$,
\begin{itemize}
\item[{\scriptsize $\circ$}] $\vphantom{f^f} x \leq r \ \, \Leftrightarrow \ \, x \leq \ell \text{ for some }\ell \in L^r$, 
\item[{\scriptsize $\circ$}] $r \leq x \ \, \Leftrightarrow \, \ u \leq x \text{ for some }u \in U^r$.
\end{itemize}
Furthermore, there are some such sets $\vphantom{f^{f^f}} L^r,U^r \in M_\a[G]$. 
\end{enumerate}
\end{theorem}
\begin{proof}
Fix $\a < \k$. 
To prove the first claim, we begin by working in $V$. 
As $|M_\xi| = \aleph_1$ for each $\xi < \a$, we have $\card{\k \cap \bigcup_{\xi < \a}M_\xi} \leq |\a| \cdot \aleph_1 < \k$, which implies $\card{\mathrm{Lim}(\k) \setminus \bigcup_{\xi < \a}M_\xi} = \k$. 
By part $(4)$ of our definition of sage Davies trees, $\seq{M_\xi}{\xi < \a} \in M_\a$, and (because $\k \in M_\a$ and $M_\a$ is closed under basic set-theoretic operations) $\mathrm{Lim}(\k) \cap \bigcup_{\xi < \a}M_\xi \in M_\a$. 
By elementarity, $M_\a \models \card{\mathrm{Lim}(\k) \setminus \bigcup_{\xi < \a}M_\xi} = \k$, and in particular there is a bijection $f \in M_\a$ from $\mathrm{Lim}(\k) \setminus \bigcup_{\xi < \a}M_\xi$ to $\k$. Restricting $f$ to $M_\a$ gives a bijection from $\mathrm{Lim}(\k) \cap M_\a \setminus \bigcup_{\xi < \a}M_\xi$ to $\k \cap M_\a$, and $|\k \cap M_\a| = \aleph_1$. (In fact, standard arguments show that $\w_1 \sub M$ whenever $M$ is a countably closed model of enough of \zfc.) 
Thus $\card{\mathrm{Lim}(\k) \cap M_\a \setminus \bigcup_{\xi < \a}M_\xi} = \aleph_1$. 

Because $M_\a$ is closed under basic set-theoretic operations, an ordinal $\lambda \in M_\a$ if and only if $\lambda+n \in M_\a$ for all $n \in \w$. Thus for $\lambda \in \mathrm{Lim}(\k)$, 
$$\textstyle \lambda \in M_\a \setminus \bigcup_{\xi < \a}M_\xi \qquad \Leftrightarrow \qquad [\lambda,\lambda+\w) \sub M_\a \setminus \bigcup_{\xi < \a}M_\xi.$$
Still working in the ground model, for each $\lambda \in \mathrm{Lim}(\k)$ let $c_\lambda$ denote the natural isomorphism $\mathrm{Fn}\big( [\lambda,\lambda+\w),2 \big) \to \mathrm{Fn}(\w,2)$. 

In $V[G]$, each $c_\lambda$ maps $G \cap \mathrm{Fn}\big( [\lambda,\lambda+\w),2 \big)$ to a generic-over-$V$ subset of $\mathrm{Fn}(\w,2)$, whose union is a Cohen-generic subset of $\w$, say $g_\lambda$. Define $C(\lambda) = g_\lambda$ for all $\lambda \in \mathrm{Lim}(\k)$. This definition of $C$ has only $G$ as a parameter, so $C \in M_\a[G]$ for all $\a < \k$, and it is clear that $C$ maps $\mathrm{Lim}(\k) \cap M_\a$ to a set of $\aleph_1$ reals mutually Cohen generic over $V\big[ G \cap \bigcup_{\xi<\a}M_\xi \big]$.

\smallskip

Before proving the second claim we must introduce a few definitions. For convenience, let $\PP = \mathrm{Fn}(\k,2)$, and given $S \sub \k$, let $\PP_S = \mathrm{Fn}(S,2)$. 
Let $G_S = G \cap \PP_S$. 
Note that if $\k \supseteq T \supseteq S$ then $\PP_T = \PP_S*\PP_{T \setminus S}$, $G_{T \setminus S}$ is $\PP_{T \setminus S}$-generic over $V[G_S]$, and $V[G_T] = V[G_S][G_{T \setminus S}]$. 
If $\dot x$ is a nice $\PP$-name for a subset of $\w$ and $S \sub \k$, define the \emph{partial evaluation} of $\dot x$ with respect ot $G_S$ to be 
$$\mathrm{pval}_S(\dot x) \,=\, \set{ \big( n,p \rest (\k \setminus S) \big) }{(n,p) \in \dot x \text{ and } p \rest S \in G_S}.$$ 
For example, if $\dot x$ is a $\PP_S$-name, then $p \rest (\k \setminus S) = \0$ whenever $(n,p) \in \dot x$, in which case 
$\mathrm{pval}_S(\dot x) \,=\, \set{(n,\0)}{(n,p) \in \dot x \text{ and } p \in G_S},$ 
which is just the canonical name $\check{y}$ for $y = \mathrm{val}_{G_S}(\dot x) = \mathrm{val}_{G}(\dot x)$. 
At the other extreme, if $\dot x$ is a $\PP_{\k \setminus S}$-name then $p \rest (\k \setminus S) = p$ and $p \rest S =  \0 \in G$ whenever $(n,p) \in \dot x$, so that $\mathrm{pval}(\dot x) = \dot x$. 
Observe that for any nice $\PP$-name $\dot x$ and any $S \sub \k$, $\mathrm{pval}_S(\dot x)$ is a nice $\PP_{\k \setminus S}$-name for a subset of $\w$, and $\0 \forces \mathrm{val}_G(\dot x) = \mathrm{val}_{G_{\k \setminus S}}\big(\mathrm{pval}_{G_S}(\dot x)\big)$. 

\smallskip

We now proceed to the proof of the second claim. 
Fix $\a < \k$. 
By part $(3)$ of the definition of a sage Davies tree, there is a countable set $\mathcal N_\a$ of countably closed elementary submodels of $H_\theta$ with $\bigcup \mathcal N_\a = \bigcup_{\xi < \a}M_\xi$.

\begin{claim} 
Let $r \in \P(\w) \cap M_\a[G] \setminus \bigcup_{\xi < \a} M_\xi[G]$, let $\dot r \in M_\a$ be a nice name for $r$, and let $N \in \mathcal N_\a$. 
For each $p \in \PP$, there is a nice name $\dot \ell \in N$ for the set 
$$\ell^r_{p,N} \,=\, \textstyle \bigcup \set{\mathrm{val}_G(\dot x)}{\dot x \text{ is a nice $\PP$-name in $N$ and } p \forces \dot x \sub \dot r}.$$ Consequently, $\ell^r_{p,N} \in N[G]$. 
\end{claim}
\noindent\emph{Proof of claim:} 
Fix $p \in \PP$. 
Working in $V$, let $X$ denote the set of all nice names $\dot x$ for subsets of $\w$ such that $\dot x \in N$ and $p \forces \dot x \sub \dot r$.  
Let 
$$\dot \ell_0 \,=\, \textstyle \bigcup X \,=\, \set{(n,q)}{\text{there is some }\dot x \in X \text{ with } (n,q) \in \dot x}.$$ 
If $\dot x \in X$ then $\dot x \in N$ and, because nice names for subsets of $\w$ are countable, this implies $\dot x \sub N$. 
Hence $\dot \ell_0 = \bigcup X \sub N$, and it is not hard to see that $\dot \ell_0$ is a $\PP$-name for $u^r_{p,N}$. 
But $\dot \ell_0$ may not be a nice name, and we have no reason to believe $\dot \ell_0 \in N$ 
(as neither $p$ nor $\dot r$ need be in $N$). 

For each $n \in \w$ let $B_n = \set{q}{(n,q) \in \dot \ell_0}$. Note that if $(n,q) \in \dot \ell_0$ then $(n,q) \in N$, hence $q \in N$; therefore $B_n \sub N \cap \PP = \PP_{\k \cap N}$. 
Let $A_n$ be an antichain in $\PP_{\k \cap N}$ such that $\bigvee A_n = \bigvee B_n$ in the Boolean completion of $\PP_{\k \cap N}$. 
Let $\dot \ell = \bigcup_{n \in \w}(\{n\} \times A_n)$. 
This is a nice $\PP_{\k \cap N}$-name for a subset of $\w$, and by our choice of the $A_n$, $\0 \forces \dot \ell_0 = \dot \ell$, which means $\dot u$ is a nice $\PP_{\k \cap N}$-name for $\ell^r_{p,N}$. 
Furthermore, because $\PP_{\k \cap N}$ has the ccc, each $A_n$ is countable. Because $N$ is countably closed, $A_n \in N$ for all $n$. Because $N$ is also closed under basic set-theoretic operations, $\{n\} \times A_n \in N$ for all $N$ and, using the countable closure of $N$ again, this implies $\dot \ell \in N$.
\hfill {\tiny $\square$}

\begin{claim} 
Let $r \in \P(\w) \cap M_\a[G] \setminus \bigcup_{\xi < \a} M_\xi[G]$. 
Suppose $N \in \mathcal N_\a$, $x \in N[G]$, and $x \sub r$. 
Let $\dot r \in M_\a$ be a nice name for $r$, and let $S = \mathrm{supp}(\dot r)$. There is a nice name $\dot x$ for $x$, with $\dot x \in N$, and a $p \in G_S$ such that $p \forces \dot x \sub \dot r$.
\end{claim}
\noindent\emph{Proof of claim:} 
Let $\dot x_0$ be a nice $\PP$-name for $x$ such that $\dot x_0 \in N$, and fix a condition $q \in G$ such that $q \forces \dot x_0 \sub \dot r$. Let $S_0 = \mathrm{supp}(\dot x_0)$, and note that $S_0 \sub N$ (because $\dot x_0 \in N$). Without loss of generality, we may (and do) suppose $\mathrm{supp}(q) \sub \mathrm{supp}(\dot x_0) \cup \mathrm{supp}(\dot r) = S_0 \cup S \sub N \cup S$. 
Let $T = \mathrm{supp}(q) \setminus S$, and note that $T$ is a finite subset of $\k \cap N$. 

Let $\dot x = \mathrm{pval}_T(\dot x_0)$ and let $p = q \rest (\k \setminus T)$. 
Clearly $p \in G$, because $q \in G$ and $q$ extends $p$. 
Because $q \forces \dot x_0 \sub \dot r$ in $V$, 
$p \forces \mathrm{pval}_T(\dot x_0) \sub \mathrm{pval}_T(\dot r)$ in $V[G_T]$. 
But $\mathrm{pval}_T(\dot x_0) = \dot x$ by definition, $\mathrm{pval}_T(\dot r) = \dot r$ because $T \cap S = \0$, and $V[G_T] = V$ because $T$ is finite. 
Hence $p \forces \dot x \sub \dot r$ in $V$. 
Finally, $\dot x$ is definable from $\dot x_0$ and $G_T$, which are both in $N$, and therefore $\dot x \in N$.
\hfill {\tiny $\square$}

\medskip

Now fix $r \in \P(\w) \cap M_\a[G] \setminus \bigcup_{\xi < \a} M_\xi[G]$. 
Let $\dot r \in M_\a$ be a nice name for $r$, and 
let $S = \mathrm{supp}(\dot r)$. 
Because $\dot r$ is a nice name for a real, $S$ is countable. 
This implies $G_S$ is countable as well. 
For each $N \in \mathcal N_\a$ and each $p \in G_S$, let $\ell^r_{p,N}$ be defined as in the statement of the claim above. Let 
$$L^r_0 \,=\, \textstyle \set{\ell^r_{p,N}}{p \in G_S \text{ and }N \in \mathcal N_\a}.$$
This is a countable set, because $G_S$ and $\mathcal N_\a$ are both countable, and it is a subset of $\P(\w) \cap \bigcup_{N \in \mathcal N_\a}N[G] = \P(\w) \cap \bigcup_{\xi < \a}M_\xi[G]$ by the first of the two claims above. 

Let $x \in \P(\w) \cap \bigcup_{\xi < \a}M_\xi[G] = \P(\w) \cap \bigcup_{N \in \mathcal N_\a}N[G]$, and fix $N \in \mathcal N$ with $x \in N[G]$. 
If $x \sub r$, then by the second of the two claims above, there is a nice name $\dot x \in N$ for $x$ and a $p \in G_S$ such that $p \forces \dot x \sub \dot r$. By the definition of $\ell^r_{p,N}$, this means $x \sub \ell^r_{p,N} \in L^r_0$. 
Conversely, if $x \sub \ell^r_{p,N}$ for some $p \in G_S$ and $N \in \mathcal N_\a$, then it is clear from the definition of $\ell^r_{p,N}$ that $x \sub r$. 

Thus, given $x \in \P(\w) \cap \bigcup_{\xi < \a}M_\xi[G]$, we have $x \sub r$ if and only if $x \sub \ell$ for some $\ell \in L^r_0$. Passing to equivalence classes, set $L^{[r]_{\mathrm{Fin}}} = \set{[\ell]_{\mathrm{Fin}}}{\ell \in L^r_0}$ has the desired properties. 
An essentially identical argument can be used to prove the analogous statement involving $U^r$. Alternatively, one can deduce the statement about $U^r$ from the statement about $L^r$ by taking complements, setting $U^r = \set{[\w]_{\mathrm{Fin}} - \ell\,}{\ell \in L^{[\w]_{\mathrm{Fin}} - r}}$. 

Finally, to prove the ``furthermore'' part of the theorem, we use the elementarity of $M_\a$. 
The part of statement $(2)$ that precedes the ``furthermore'' is a first-order description of $L^r$ and $U^r$ using only parameters from $M_\a[G]$. (Note that $\bigcup_{\xi < \a}M_\a[G] \in M_\a[G]$ because $\seq{M_\xi}{\xi < \a} \in M_\a$ and $G \in M_\a[G]$). 
Because $M_\a \preceq H_\theta$, the definability of the forcing relation implies $M_\a[G] \preceq H_\theta[G]$. (For more details on this, see \cite[Theorem 2.11]{Shelah2}.) 
Thus this description of $L^r$ and $U^r$, which we just proved is witnessed in $H_\theta$, is witnessed in $M_\a[G]$ as well.
\end{proof}

\section{Building the automorphisms}

An \emph{almost permutation} on $\w$ is a bijection between cofinite subsets of $\w$. 
Every permutation of $\w$ is an almost permutation, but not every almost permutation is a permutation or even the restriction of one (e.g. consider the successor function $n \mapsto n+1$). 
Every almost permutation $f$ of $\w$ induces an automorphism $\alpha_f$ of $\pwmf$, defined by the formula 
$$\alpha_f \big( [A]_{\mathrm{Fin}} \big) \,=\, \big[ f[A] \big]_{\mathrm{Fin}}$$
for every $A \sub \w$.  
An automorphism of $\pwmf$ that is induced in this way by an almost permutation of $\w$ is called \emph{trivial}. 
Note that as there are only $\continuum$ functions $\w \to \w$, there are only $\continuum$ trivial automorphisms of $\pwmf$. Consequently, if there are more than $\continuum$ automorphisms of $\pwmf$, then some (most) are nontrivial. 

Walter Rudin proved in \cite{Rudin} that \ch implies there are $2^\continuum$ automorphisms of $\pwmf$. 
Years later, Shelah proved in \cite[Chapter 4]{Shelah} that it is consistent with \zfc that all automorphisms are trivial. Shelah's result was extended and refined in \cite{SS}, \cite{Velickovic}, and \cite{DFV}, revealing ultimately that the forcing axiom $\mathsf{OCA}$ (also known as $\mathsf{OGA}$ or $\mathsf{OCA}_T$) implies all automorphisms are trivial. 
Further models in which all automorphisms are trivial include models with $\continuum > \aleph_2$ in \cite{Dow} or with $\dom = \aleph_1$ in \cite{FarSh}. (Recall that \oca implies $\dom = \continuum = \aleph_2$). 

The following theorem covers parts $(1)$ and $(2)$ of the main theorem stated in the introduction. 

\begin{theorem}\label{thm:main}
Suppose $V \models \ch$, let $\k \geq \aleph_2$ be a regular cardinal, and let $G$ be a $\mathrm{Fn}(\k,2)$-generic filter over $V$. 
Assume that for any $P$ there is a sage Davies tree of length $\k$ over $P$. 
If $\phi$ is an automorphism of $\pwmf$ in $V$, then in $V[G]$ there are $2^\k$ automorphisms of $\pwmf$ extending $\phi$. 
\end{theorem}
\begin{proof}
Let $\phi$ be an automorphism of $\pwmf$ in $V$, and (still in $V$) fix a well ordering $R$ of $H_\k$. 
Let $F$ be a set of $2^\k$ functions $\k \to 2$ with the property that if $h,h' \in F$ and $h \neq h'$, then $\set{\a \in \k}{h(\a) \neq h'(\a)}$ is stationary. 
(To see that there is such a set of functions, first recall that there is a partition $\P$ of $\k$ into $\k$ stationary sets. Take $F$ to be the set of all functions $h: \k \to 2$ such that $h \rest A$ is constant for every $A \in \P$.)

To prove the theorem, we construct for every $h \in F$ an automorphism $\Phi_h$ of $\pwmf$ in $V[G]$ such that $\Phi_h \rest V = \phi$, and furthermore, we do this in such a way that $h \neq h'$ implies $\Phi_h \neq \Phi_{h'}$. 
Fix $h \in F$, and let $\seq{M_\a}{\a < \k}$ be a sage Davies tree over $P = (R,h,\phi)$. 

Let $C$ denote the function described in Theorem~\ref{thm:firewood}$(1)$. 
For each $\a < \k$, let $\dlt_\a = \min \big( \k \cap M_\a \setminus \bigcup_{\xi < \a} \big)$. 
Note that $\dlt_\a$ is a limit ordinal, and because $\dlt_\a$ and $\dlt_\a+\w$ are definable from each other, $\dlt+\w \in \k \cap M_\a \setminus \bigcup_{\xi < \a}M_\xi$ as well. 
Define $c^\a_0 = C(\dlt_\a)$ and $c^\a_1 = C(\dlt_\a+\w)$ for all $\a < \k$.

Define $B_0 = \pwmf \cap V$, and if $0 < \a < \k$ then define 
$$B_\a \,=\, \textstyle \big\<\hspace{-1.2mm}\big\< \pwmf \cap \bigcup_{\xi < \a}M_\xi[G] \,\big\>\hspace{-1.2mm}\big\>,$$
the subalgebra of $\pwmf$ generated by $\pwmf \cap \bigcup_{\xi < \a}M_\xi[G]$. 
The fact that $M_0$ is countably closed and $M_0 \preceq H_\theta$ implies $\P(\w) \cap V \sub M_0$. 
Hence $B_0 \sub B_1$, and it is clear that if $\a < \b < \k$ then $B_\a \sub B_\b$. 
Thus $\seq{B_\a}{\a < \k}$ is an increasing sequence of subalgebras of $\pwmf$, and (it follows easily from the definition that) $B_\a = \bigcup_{\xi < \a}B_\xi$ for limit $\alpha$.  

We claim furthermore that $\pwmf = \bigcup_{\a < \k}B_\a$. 
To see this, it suffices to show, working in $V$, that every nice $\PP$-name for a subset of $\w$ is a member of $\bigcup_{\a < \k}M_\a$. If $\dot x$ is such a name, then $\mathrm{supp}(\dot x)$ is a countable subset of $\k$. 
By part $(2)$ of the definition of a sage Davies tree, there is some $\a < \k$ with $\mathrm{supp}(\dot x) \in M_\a$. 
Using the fact that $M_\a$ is countably closed and $\w \sub M_\a$, it follows that $\dot x \in M_\a$. 

Suppressing the subscript for now, let us construct from $h$ an automorphism $\Phi$ of $\pwmf$ in $V[G]$. 
We aim to build $\Phi$ by transfinite recursion, working in $V[G]$. 
The recursion is organized into $\k$ main stages. 
At stage $\a$ of the recursion we begin with an automorphism $\Phi_\a$ of $B_\a$, and we wish to extend it to an automorphism $\Phi_{\a+1}$ of $B_{\a+1}$ while also encoding the value of $h(\a)$ into $\Phi_{\a+1}$. 
Through the recursion we maintain the following inductive hypothesis at stage $\a$: 
\begin{itemize}
\item[$(*)$] the sequence of automorphisms $\seq{\Phi_\xi}{\xi \leq \a}$ is uniformly definable from $P$ and $\seq{M_\xi[G]}{\xi < \a}$. In particular $\Phi_\a \in M_\a$. 
\end{itemize}

For the base case $\a = 0$ we begin with the automorphism $\Phi_0 = \phi$ of $B_0$. 
Note that our inductive hypothesis is trivially satisfied at $\a=0$. 
At limit stages we simply take $\Phi_\a = \bigcup_{\xi < \a}\Phi_\xi$. 
This is an automorphism of $B_\a$ because $B_\a = \bigcup_{\xi < \a}B_\xi$ for limit $\alpha$. And this preserves our inductive hypothesis: $\seq{\Phi_\xi}{\xi \leq \a}$ is defined as $\seq{\Phi_\xi}{\xi < \a} \cat \< \bigcup_{\xi < \a}\Phi_\xi \>$, and this defining formula is uniform across all limit ordinals.

We must now describe how $\Phi_{\a+1}$ is obtained from $\Phi_\a$. This single stage of the main recursion is itself accomplished by a transfinite recursion of length $\w_1$. 
Observe that $M_\a \in M_{\a+1}$ (it is definable as the last element of $\seq{M_\xi}{\xi \leq \a} \in M_{\a+1}$), which means that $M_{\a+1}$ contains a canonical well ordering, given by $R$, of the nice names for subsets of $\w$ in $M_\a$. 
Evaluating these names, and deleting repeats and members of $B_\a$ (which is also definable in $M_{\a+1}$), this means $M_{\a+1}$ contains a canonical enumeration $\seq{b_\xi}{\xi < \w_1}$ of the members of $\pwmf \cap M_\a[G] \setminus B_\a$. 
Making some simple modifications to this well order if needed, we may (and do) assume 
$b_0 = c^\a_0$. 
The goal of our stage $\a$, length-$\w_1$ sub-recursion is to construct
\begin{enumerate}
\item two increasing sequences $\seq{D_\a^\b}{\b < \w_1}$ and $\seq{E_\a^\b}{\b < \w_1}$ of subalgebras of $B_{\a+1}$, such that $D_\a^\b = \bigcup_{\xi < \b}D_\a^\xi$ and $E_\a^\b = \bigcup_{\xi < \b}E_\a^\xi$ for limit $\b$, and $\bigcup_{\b < \w_1}D_\a^\b = \bigcup_{\b < \w_1}E_\a^\b = B_{\a+1}$;
\item a sequence of injective homomorphisms $\Phi_\a^\b: D_\a^\b \to E_\a^\b$, for $\b < \w_1$, such that $\Phi_\a^0 = \Phi_\a$ and if $\b < \b'$ then $\Phi_\a^\b = \Phi_\a^{\b'} \rest D_\a^\b$.
\end{enumerate}
Furthermore, at each step of this sub-recursion on $\b$ we maintain the following inductive hypotheses:  
\begin{enumerate}
\item $D_\a^\b,E_\a^\b \supseteq \<\hspace{-.9mm}\< B_\a \cup \set{b_\xi}{\xi < \b} \>\hspace{-.9mm}\>$.  
\item There is a countable $X^\b \sub M_\a[G]$ such that $D_\a^\b = \<\hspace{-.9mm}\< B_\a \cup X^\b \>\hspace{-.9mm}\>$. 
\item Furthermore, the sequence $\seq{X^\xi}{\xi \leq \b}$ is uniformly definable from parameters in $M_\a[G]$, hence it is a member of $M_\a[G]$. 
\item The sequence $\seq{\Phi_\a^\xi}{\xi \leq \b}$ is also uniformly definable from parameters in $M_\a[G]$, hence a member of $M_\a[G]$.
\end{enumerate} 

For the base step, take $D_\a^0 = E_\a^0 = B_\a$ and $\Phi_\a^0 = \Phi_\a$. 
This trivially satisfies hypotheses $(1)-(3)$. It satisfies $(4)$ because $\Phi_\a \in M_\a[G]$ by our larger inductive hypothesis $(*)$.

The limit steps are also easy: because $D_\a^\b = \bigcup_{\xi < \b}D_\a^\xi$ and $E_\a^\b = \bigcup_{\xi < \b}E_\a^\xi$ for limit $\b$, we simply take $\Phi_\a^\b = \bigcup_{\xi < \b}\Phi_\a^\xi$. 
This preserves all four of the inductive hypotheses. 

To obtain $\Phi^1_\a$, we do one of two things, depending on the value of $h(\a)$. 
If $h(\a)=0$ then we map $b_0 = c^\a_0$ to itself. 
Because $c^\a_0$ is Cohen-generic over $V[G \cap \bigcup_{\xi < \a}M_\a]$, it is free over $B_\a$.  Consequently, the map $\Phi^0_\a \cup \{(c^\a_0,c^\a_0)\}$ naturally extends to an automorphism $\Phi^1_\a$ of $\<\hspace{-.9mm}\< B_\a \cup \{c^\a_0\} \>\hspace{-.9mm}\>$, and we set $D^1_\a = E^1_\a = \<\hspace{-.9mm}\< B_\a \cup \{c^\a_0\} \>\hspace{-.9mm}\>$. 
If $h(\a)=1$, then we map $c^\a_0$ to $c^\a_1$ and $c^\a_1$ to $c^\a_0$. 
Because $c^\a_0$ and $c^\a_1$ are mutually Cohen-generic over $V[G \cap \bigcup_{\xi < \a}M_\a]$, each $c^\a_i$ is free over $\<\hspace{-.9mm}\< B_\a \cup \{c^\a_{1-i}\} \>\hspace{-.9mm}\>$. 
Consequently, the map $\Phi^0_\a \cup \{(c^\a_0,c^\a_1),(c^\a_1,c^\a_0)\}$ naturally extends to an automorphism $\Phi^1_\a$ of $\<\hspace{-.9mm}\< B_\a \cup \{c^\a_0,c^\a_1\} \>\hspace{-.9mm}\>$, and we set $D^1_\a = E^1_\a = \<\hspace{-.9mm}\< B_\a \cup \{c^\a_0,c^\a_1\} \>\hspace{-.9mm}\>$.
This clearly preserves the inductive hypotheses (bearing in mind, for inductive hypothesis $(4)$, that $M_\a[G]$ is closed under the Boolean operations of $\pwmf$). 

For the general successor step, fix $\b$ with $0 < \b < \w_1$, and let $D_\a^\b$, $E_\a^\b$, and $\Phi_\a^\b$ be given satisfying $(1)-(4)$. 
We obtain $\Phi_\a^{\b+1}$ in two steps, first extending our algebras and mappings so that $b_\b$ gets into the domain of $\Phi_\a^{\b+1}$, and extending again to put $b_\b$ into the image. 
Specifically, we first extend $D_\a^\b, E_\a^\b, \Phi_\a^\b$ to some $\bar D_\a^\b, \bar E_\a^\b, \bar \Phi_\a^\b$ with $b_\b \in \bar D_\a^\b$, 
and then we extend $\bar D_\a^\b, \bar E_\a^\b, \bar \Phi_\a^\b$ to $D_\a^{\b+1}, E_\a^{\b+1}, \Phi_\a^{\b+1}$ with $b_\b \in E_\a^{\b+1}$. 

If $b_\b \in D_\a^\b$ already, then there is nothing to do in the first step: simply define $\bar D_\a^\b = D_\a^\b$, $\bar E_\a^\b = E_\a^\b$, and $\bar \Phi_\a^\b = \Phi_\a^\b$. 

Suppose instead that $b_\b \notin D_\a^\b$. 
Applying Theorem~\ref{thm:firewood}, fix countable sets $L,U \sub \pwmf \cap \bigcup_{\xi < \a}M_\xi[G]$ with $L,U \in M_\a[G]$ having the properties stated in part $(2)$ of the theorem with $r=b_\b$. 
Let $X^\b$ denote the set described in inductive hypotheses $(2)$ and $(3)$,  
and define 
\begin{align*}
\tilde L &\,=\, \set{\ell \in \<\hspace{-.9mm}\< L \cup X^\b \>\hspace{-.9mm}\> }{\ell < b_\b}, \\
\tilde U &\,=\, \set{\ell \in \<\hspace{-.9mm}\< U \cup X^\b \>\hspace{-.9mm}\> }{b_\b < u}.
\end{align*} 
These are countable subsets of $D_\a^\b$ and, using the properties of $L$ and $U$ stated in Theorem~\ref{thm:firewood}, if $x \in D_\a^\b$ then
\begin{itemize}
\item[{\scriptsize $\circ$}] $x \leq b_\b \ \, \Leftrightarrow \ \, x \leq \ell \text{ for some }\ell \in \tilde L$, 
\item[{\scriptsize $\circ$}] $b_\b \leq x \ \, \Leftrightarrow \, \ u \leq x \text{ for some }u \in \tilde U$.
\end{itemize} 
In other words, $(\tilde L,\tilde U)$ is an $(\w,\w)$-gap in $D_\a^\b$, and this gap is split by $b_\b$. 
It follows that 
$\big( \Phi_\a^\b[\tilde L],\Phi_\a^\b[\tilde U] \big)$ is an $(\w,\w)$-gap in $E_\a^\b$. 
If we are to extend $\Phi_\a^\b$ to $b_\b$, we must map $b_\b$ to a member of $M_\a[G]$ that splits this gap. 

Because $L,U,X^\b,b_\b \in M_\a[G]$, we have $\tilde L,\tilde U \in M_\a[G]$. 
By inductive hypothesis $(4)$, we have $\Phi_\a^\b[\tilde L],\Phi_\a^\b[\tilde U] \in M_\a[G]$ as well. 
Because $X^\b$ is countable and $D_\a^\b = \<\hspace{-.9mm}\< B_\a \cup X^\b \>\hspace{-.9mm}\>$, and because the members of the (uncountable) set $C_\a$ are all Cohen-generic over $B_\a$, all but countably many members of $C_\a$ are Cohen-generic over $D_\a^\b$. 
Fix some such $c \in C_\a$. 

The standard forcing poset used to split an $(\w,\w)$-gap in $\pwmf$ is countable, hence forcing equivalent to the Cohen forcing $\mathrm{Fn}(\w,2)$. 
Because $c$ is Cohen-generic over $E_\a^\b$, and because $M_\a[G]$ satisfies (enough of) \zfc, this means that, working within $M_\a[G]$, it is possible to define a generic splitting $s_\b \in M_\a[G]$ of the $(\w,\w)$-gap 
$\big( \Phi_\a^\b[\tilde L],\Phi_\a^\b[\tilde U] \big)$ in $E_\a^\b$. 
Furthermore, we may (and do) make the choice of $s_\b$ canonical by taking the $R$-least $c \in C_\a$ (i.e., the one with the $R$-least name) generic over $D^\b_\a$, the $R$-least dense embedding between $\mathrm{Fn}(\w,2)$ and $\big( \Phi_\a^\b[\tilde L],\Phi_\a^\b[\tilde U] \big)$, etc. 
Because $s_\b$ splits the gap $\big( \Phi_\a^\b[\tilde L],\Phi_\a^\b[\tilde U] \big)$, the function $\Phi_\a^\b \cup \{(b_\b,s_\b)\}$ extends to an isomorphism from $\bar D^\b_\a = \<\hspace{-.9mm}\< D_\a^\b \cup \{b_\b\} \>\hspace{-.9mm}\>$ to $\bar E^\b_\a = \<\hspace{-.9mm}\< E_\a^\b \cup \{s_\b\} \>\hspace{-.9mm}\>$. 
Let $\bar \Phi_\a^\b$ denote this isomorphism. 

The construction of $D_\a^{\b+1}$, $E_\a^{\b+1}$, and $\Phi_\a^{\b+1}$ from $\bar D_\a^{\b}$, $\bar E_\a^{\b}$, and $\bar \Phi_\a^{\b}$ works in exactly the same way, only using $\bar E_\a^\b$ in the place of $\bar D_\a^b$, $(\bar \Phi_\a^{\b})^{-1}$ in the place of $\Phi_\a^{\b}$, etc.

These choices clearly preserve inductive hypotheses $(1)$ and $(2)$. 
Because all the choices in this construction were made canonically via $R$, the preceding construction constitutes a definition of $D_\a^{\b+1}$, $E_\a^{\b+1}$, and $\Phi_\a^{\b+1}$ from $D_\a^{\b}$, $E_\a^{\b}$, and $\Phi_\a^{\b}$. 
Moreover this definition is uniform, as it works the same way for all $\b$. 
It follows that $(3)$ and $(4)$ are preserved. 

This concludes our length-$\w_1$ sub-recursion. Inductive hypothesis $(1)$ makes it clear that $\bigcup_{\b < \w_1} D^\b_\a = \bigcup_{\b < \w_1} E^\b_\a = B_{\a+1}$, and the fact that $\Phi^\b_\a = \Phi^{\b'}_\a \rest D_\a^\b$ whenever $\b < \b'$ shows that $\Phi_{\a+1} = \bigcup_{\b < \w_1}\Phi^\b_\a$ is an automorphism of $B_{\a+1}$. 
Finally, the entire sub-recursion was defined from $M_\a[G]$, $C$, and $R$ (or $P$). 
Applying the Recursion Theorem in $M_{\a+1}[G]$, and using inductive hypothesis $(4)$, this means that our entire length-$\w_1$ sub-recursion, and its result $\Phi_{\a+1}$, are definable from $\seq{M_\xi[G]}{\xi \leq \a}$ and $R$ in $M_{\a+1}[G]$. Furthermore, this definition is uniform across various $\a$ (i.e., our description of the inductive step does not depend on $\a$), so that $(*)$ is preserved at successor steps of the main recursion. 

This completes the recursion, the result of which is an automorphism $\Phi = \bigcup_{\a < \k}\Phi_\a$ of $\pwmf = \bigcup_{\a < \k}B_\a$ in $V[G]$ that extends $\phi = \Phi_0$. 
It remains to show that there are $2^\k$ automorphisms like this. 

Recall that our notation has been slightly dishonest: the construction above depends on a function $h \in F$, although we redacted $h$ from $\Phi$ for the sake of tidiness. 
Restoring our subscript notation from the beginning of the proof, for every $h \in F$ the preceding construction gives an automorphism $\Phi_h$ of $\pwmf$ in $V[G]$ extending $\phi$. 
To finish the proof, we must show that if $h,h' \in F$ and $h \neq h'$, then $\Phi_h \neq \Phi_h'$. 

Given our construction, one might be tempted to think that $h$ can be recovered from $\Phi_h$: if $\Phi_h(c^\a_0) = c^\a_0$ then $h(\a) = 0$, and if $\Phi_h(c^\a_0) \neq c^\a_0$ then $h(\a) = 1$. 
However, the value of $c^\a_0$ is determined note only by $C$, which does not depend on $h$, but also $M_\a \setminus \bigcup_{\xi < \a}M_\xi$, which does depend on $h$. 
Recall that $h$ was taken as part of the parameter $P$ for the sage Davies tree $\seq{M_\a}{\a < \k}$ used to obtain $\Phi_h$, which means that the identity of the $c^\a_0$'s depends on $h$. 
Thus the proposed definition of $h$ from $\Phi_h$ is circular.

To get around this, we use part $(6)$ of our definition of a sage Davies tree. 
Let $\seq{M_\a^h}{\a < \k}$ denote the sage Davies tree used to obtain $\Phi_h$, and 
let $\big\< M_\a^{h'} :\, \a < \k \big\>$ denote the sage Davies tree used to obtain $\Phi_{h'}$. 
Similarly, let 
$(c^\a_0)^h \!= C\big(\! \min (\k \cap M_\a^h \setminus \bigcup_{\xi < \a}M_\xi^h) \big)$ and 
$(c^\a_0)^{h'} \!= C\big(\! \min (\k \cap M_\a^{h'} \setminus \bigcup_{\xi < \a}M_\xi^{h'}) \big)$ for all $\a < \k$. 
Applying part $(6)$ of our definition of sage Davies trees, 
$$D \,=\, \set{\dlt \in \k}{\dlt = \k \cap \textstyle \bigcup_{\xi < \dlt}M_\xi^h = \k \cap \textstyle \bigcup_{\xi < \dlt}M_\xi^{h'}}$$
is closed and unbounded in $\k$. 
If $\dlt \in D$, then $(c^\dlt_0)^h = (c^\dlt_0)^{h'} = C(\dlt)$. 

By our choice of $F$, the set $S = \set{\a < \k}{h(\a) \neq h'(\a)}$ is stationary. In particular, there is some $\dlt \in D$ with $h(\dlt) \neq h'(\dlt)$. 
By the previous paragraph, $(c^\dlt_0)^h = (c^\dlt_0)^{h'} = C(\dlt)$. 
But since $h(\dlt) \neq h'(\dlt)$, one of the functions $\Phi_h$ and $\Phi_{h'}$ has $C(\dlt)$ as a fixed point, and the other does not. Hence $\Phi_h \neq \Phi_{h'}$. 
\end{proof}

\begin{proof}[Proof of part $(3)$ of the main theorem]
We now describe how to modify the preceding argument in order to accommodate singular $\k$. 

First consider the case $\k^\w = \k$. In this case Theorem~\ref{thm:firewood} applies, and the proof of the previous theorem goes through unmodified except for the last 2 paragraphs, where we use part $(6)$ of the definition of a sage Davies tree (the part of the definition that assumes $\k$ is regular) in order to prove the $\Phi_h$ are all distinct. When $\k$ is singular we do not aim to prove all the $\Phi_h$ are distinct, but only that at least one of them is a nontrivial automorphism. 

Consider the automorphism $\Phi_h$ constructed from the constant function $h(\a) = 1$. 
We claim that this function is a nontrivial automorphism of $\pwmf$. 
To see this, fix an almost bijection $f: \w \to \w$. 
In the ground model $V$, there is a nice name $\dot f$ for $f$ with countable support $S$. 
Setting $G_S = G \cap \mathrm{Fin}(S,2)$, we have $f \in V[G_S]$. 
The reals $\set{c^\a_\xi}{\a < \k, \xi < \w_1}$ are mutually Cohen-generic over $V$, and it follows that for all but countably many values of $\a$, the reals $c_\a^0,c_\a^1$ are mutually Cohen-generic over $V[G_S]$. 
In particular, this implies that $f[c^0_\a] \neq^* c^1_\a$. Because $\Phi_h(c^0_\a) = c^1_\a$, this means that $\Phi_h$ is not induced by $f$. 
As $f$ was an arbitrary almost bijection of $\w$, this shows $\Phi_h$ is nontrivial. 

Next consider the case $\k^\w > \k$. In this case we must use a sage Davies tree of length $\nu = \k^\w$ rather than length $\k$. 
Also in this case, the proof of Theorem~\ref{thm:firewood}(1) does not work, although Theorem~\ref{thm:firewood}(2) still goes through. 

To work around not having access to Theorem~\ref{thm:firewood}(1), we use a generalized almost disjoint family in $[\k]^\w$. First, note that the set $\k^{<\w}$ of all finite sequences in $\k$ is a set of size $\k$. Associating to each infinite sequence $s \in \k^\w$ its set of restrictions $A_s = \set{s \rest n}{n \in \w}$, we see that there is a size-$\k^\w$ family $\set{A_s}{s \in \k^\w}$ of pairwise almost disjoint countable subsets of $\k^{<\w}$. 

Fix a bijection $\psi$ between $\nu = \k^\w$ and some size-$\k^\w$ family of pairwise almost disjoint countable subsets of $\k$. 
Now modify the proof of Theorem~\ref{thm:main} by including $\psi$ in the parameter for our sage Davies tree: that is, take $P = (R,h,\phi,\psi)$ instead of just $P = (R,h,\phi)$.

Using the fact that $\nu > \k$, we can show (using an argument essentially identical to the first paragraph of the proof of Theorem~\ref{thm:firewood}) that $\card{\nu \cap M_\a \setminus \bigcup_{\xi < \a}M_\xi} = \aleph_1$ for each $\a < \nu$. Because $\psi \in M_\a$, the $\aleph_1$-many new ordinals $<\!\nu$ give us access to $\aleph_1$-many new countable subsets of $\k$, namely the sets $\psi(\z)$ where $\z \in \nu \cap M_\a \setminus \bigcup_{\xi < \a}M_\xi$. 

Now suppose $S \in M_\xi$ for some $\xi < \a$ and $S \in [\k]^\w$. 
Because $\psi$ maps to an almost disjoint family, $\mathrm{Hit}(S) = \set{\z < \nu}{\psi(\z) \cap S \text{ is infinite}}$ is a countable subset of $\nu$ definable from $S$. 
Consequently, $\mathrm{Hit}(S) \sub M_\xi$. 
Furthermore, if $\z \in \nu \cap M_\a \setminus \bigcup_{\xi < \a}M_\xi$, then $\psi(\z) \cap S$ is finite for all $S \in [\k]^\w \cap \bigcup_{\xi < \a}M_\xi$. 
In particular, $C_\a = \set{\psi(\z)}{\z \in \nu \cap M_\a \setminus \bigcup_{\xi < \a}M_\xi}$ is a size-$\aleph_1$ collection of countable subsets of $\k$ in $M_\a$, each one of which is almost disjoint from every such set in $\bigcup_{\xi < \a}M_\xi$. 

These sets in $C_\a$ take the role of the sets $[\lambda,\lambda+\w)$ in the proof of Theorem~\ref{thm:firewood}$(1)$. Because they are almost disjoint from each other and from the sets in $\bigcup_{\xi < \a}M_\xi$, they give rise to a set of $\aleph_1$ mutually generic Cohen reals in $M_\a[G]$ that can be used to carry out the construction of $\Phi_h$ at stage $\a$. 
As with the $\k^\w = \k$ case at the start of this proof, going through this construction with $h$ the constant function $\alpha \mapsto 1$ yields a nontrivial automorphism of $\pwmf$.
\end{proof}

We conclude with some further observations and open questions. 
The most pressing question raised by our main theorem is:

\begin{question}
Can parts $(2)$ and $(3)$ of the main theorem be proved from \zfc? And can the number of automorphisms in $(3)$ be raised to $2^\k$?
\end{question}

\noindent It is worth pointing out that if it is consistent that all automorphisms of $\pwmf$ are trivial in the $\k$-Cohen model for some $\k \geq \aleph_\w$, then proving this will require a large cardinal hypothesis. 

As was mentioned already, the hypothesis of the theorem, that there is a sage Davies tree of length $\k$, holds provided that $\mathsf{SCH}$ and $\square_\lambda$ hold at all $\lambda \leq \k$ with $\mathrm{cf}(\lambda) = \k$. 
The failure of $\square_\lambda$ at a singular $\lambda$ while $\mathsf{SCH}$ holds is quite high: \cite[Corollary 5]{SZ} shows that it gives rise to an inner model with infinitely many Woodin cardinals. (And even without $\mathsf{SCH}$, one still gets an inner model with a Woodin cardinal.) 
The failure of $\mathsf{SCH}$ has much lower consistency strength: by work of Gitik in \cite{G1,G2}, the failure of $\mathsf{SCH}$ is equiconsistent with the existence of a measurable cardinal $\mu$ with Mitchell order $\mu^{++}$. This, then, is the bottleneck for constructing sage Davies trees: if there is a model without long sage Davies trees, then it is consistent to have (at least) measurable cardinals of large Mitchell order. 

If we assume \gch in the ground model rather than merely \ch, then $\mathsf{SCH}$ holds automatically, so a lack of sage Davies trees would entail the failure of $\square_\lambda$ for singular $\lambda$. 
Furthermore, it is proved in \cite{BDMY} that under \gch, the ``Very Weak Square" of Foreman and Magidor from \cite{FM} suffices to construct sage Davies trees. 


\begin{question}
Is it consistent, relative to some large cardinal hypothesis, that all automorphisms of $\pwmf$ are trivial in the $\k$-Cohen model for some $\k \geq \aleph_\w$? 
Is it consistent (relative to large cardinals) that there is a regular $\k > \aleph_\w$ such that not all ground model automorphisms of $\pwmf$ extend to $2^\k$ automorphisms in the $\k$-Cohen model?
\end{question}

%

Finally, let us point out that while we have a decent understanding of automorphisms of $\pwmf$ after adding Cohen reals, it is open what happens after adding random reals.

\begin{question}
Are there nontrivial automorphisms of $\pwmf$ in the random real model? What about the $\k$-random real model for $\k \geq \aleph_3$?
\end{question}



\end{document}